\documentclass[12pt, reqno]{amsart}
\usepackage{amsmath, amsthm, amscd, amsfonts, amssymb, graphicx, color}
\usepackage[bookmarksnumbered, colorlinks, plainpages]{hyperref}
\hypersetup{colorlinks=true,linkcolor=red, anchorcolor=green, citecolor=cyan, urlcolor=red, filecolor=magenta, pdftoolbar=true}

\newtheorem{theorem}{Theorem}[section]
\newtheorem{lemma}[theorem]{Lemma}

\newtheorem{corollary}[theorem]{Corollary}
\theoremstyle{definition}
\newtheorem{definition}[theorem]{Definition}
\newtheorem{example}[theorem]{Example}

\theoremstyle{remark}

\numberwithin{equation}{section}

\begin{document}
\setcounter{page}{1}
\title[Continuous $\ast$-g-Frame in Hilbert $C^{\ast}$-Modules]{Continuous $\ast$-g-Frame in Hilbert $C^{\ast}$-Modules}

	\author{Mohamed ROSSAFI}
	\address{Department of Mathematics
		\newline \indent University of Ibn Tofail
		\newline \indent B.P.133 Kenitra, Morocco}
	\email{rossafimohamed@gmail.com}

\author{ Samir KABBAJ}
\address{Department of Mathematics
	\newline \indent University of Ibn Tofail
	\newline \indent B.P.133 Kenitra, Morocco}
\email{samkabbaj@yahoo.fr }

%\date{../10/2017}


\begin{abstract}
	In this paper, we introduce the concept of Continuous $\ast$-g-Frame in Hilbert $C^{\ast}$-Modules and we establish some results. We also discuss the stability problem for Continuous $\ast$-g-Frame.
\end{abstract} \maketitle

\section{Introduction and preliminaries}
The concept of frames in Hilbert spaces has been introduced by
Duffin and Schaeffer \cite{Duf} in 1952 to study some deep problems in nonharmonic Fourier
series, after the fundamental paper \cite{13} by Daubechies, Grossman and Meyer, frame
theory began to be widely used, particularly in the more specialized context of wavelet
frames and Gabor frames \cite{Gab}.

Traditionally, frames have been used in signal processing, image processing, data compression
and sampling in sampling theory. A discreet frame is a countable family of
elements in a separable Hilbert space which allows for a stable, not necessarily unique,
decomposition of an arbitrary element into an expansion of the frame elements. The
concept of a generalization of frames to a family indexed by some locally compact space
endowed with a Radon measure was proposed by G. Kaiser \cite{15} and independently by Ali,
Antoine and Gazeau \cite{11}. These frames are known as continuous frames. Gabardo and
Han in \cite{14} called these frames frames associated with measurable spaces, Askari-Hemmat,
Dehghan and Radjabalipour in \cite{12} called them generalized frames and in mathematical
physics they are referred to as coherent states \cite{11}. 

In this paper, we introduce the notion of Continuous $\ast$-g-Frame which are generalization of $\ast$-g-Frame in Hilbert $C^{\ast}$-Modules introduced by A.Alijani \cite{Ali} and we establish some results.

The paper is organized as follows, we continue this introductory section we briefly recall the definitions and basic properties of $C^{\ast}$-algebra, Hilbert $C^{\ast}$-modules. In Section 2, we introduce the Continuous $\ast$-g-Frame, the Continuous pre-$\ast$-g-frame operator and the Continuous $\ast$-g-frame operator. In Section 3, we discuss the stability problem for Continuous $\ast$-g-Frame.

In the following we briefly recall the definitions and basic properties of $C^{\ast}$-algebra, Hilbert $\mathcal{A}$-modules. Our reference for $C^{\ast}$-algebras is \cite{{Dav},{Con}}. For a $C^{\ast}$-algebra $\mathcal{A}$ if $a\in\mathcal{A}$ is positive we write $a\geq 0$ and $\mathcal{A}^{+}$ denotes the set of positive elements of $\mathcal{A}$.
	\begin{definition}\cite{Con}.
	If $\mathcal{A}$ is a Banach algebra, an involution is a map $ a\rightarrow a^{\ast} $ of $\mathcal{A}$ into itself such that for all $a$ and $b$ in $\mathcal{A}$ and all scalars $\alpha$ the following conditions hold:
	\begin{enumerate}
		\item  $(a^{\ast})^{\ast}=a$.
		\item  $(ab)^{\ast}=b^{\ast}a^{\ast}$.
		\item  $(\alpha a+b)^{\ast}=\bar{\alpha}a^{\ast}+b^{\ast}$.
	\end{enumerate}
\end{definition}
\begin{definition}\cite{Con}.
	A $\mathcal{C}^{\ast}$-algebra $\mathcal{A}$ is a Banach algebra with involution such that :$$\|a^{\ast}a\|=\|a\|^{2}$$ for every $a$ in $\mathcal{A}$.
\end{definition}
\begin{example}
	$ \mathcal{B}=B(\mathcal{H}) $	the algebra of bounded operators on a Hilbert space, is a  $\mathcal{C}^{\ast}$-algebra, where for each operator $A$, $A^{\ast}$ is the adjoint of $A$.
\end{example}
\begin{definition}\cite{Kap}.
	Let $ \mathcal{A} $ be a unital $C^{\ast}$-algebra and $\mathcal{H}$ be a left $ \mathcal{A} $-module, such that the linear structures of $\mathcal{A}$ and $ \mathcal{H} $ are compatible. $\mathcal{H}$ is a pre-Hilbert $\mathcal{A}$-module if $\mathcal{H}$ is equipped with an $\mathcal{A}$-valued inner product $\langle.,.\rangle :\mathcal{H}\times\mathcal{H}\rightarrow\mathcal{A}$, such that is sesquilinear, positive definite and respects the module action. In the other words,
	\begin{itemize}
		\item [(i)] $ \langle x,x\rangle\geq0 $ for all $ x\in\mathcal{H} $ and $ \langle x,x\rangle=0$ if and only if $x=0$.
		\item [(ii)] $\langle ax+y,z\rangle=a\langle x,y\rangle+\langle y,z\rangle$ for all $a\in\mathcal{A}$ and $x,y,z\in\mathcal{H}$.
		\item[(iii)] $ \langle x,y\rangle=\langle y,x\rangle^{\ast} $ for all $x,y\in\mathcal{H}$.
	\end{itemize}	 
\end{definition}
For $x\in\mathcal{H}, $ we define $||x||=||\langle x,x\rangle||^{\frac{1}{2}}$. If $\mathcal{H}$ is complete with $||.||$, it is called a Hilbert $\mathcal{A}$-module or a Hilbert $C^{\ast}$-module over $\mathcal{A}$. For every $a$ in $C^{\ast}$-algebra $\mathcal{A}$, we have $|a|=(a^{\ast}a)^{\frac{1}{2}}$ and the $\mathcal{A}$-valued norm on $\mathcal{H}$ is defined by $|x|=\langle x, x\rangle^{\frac{1}{2}}$ for $x\in\mathcal{H}$.

The following lemmas will be used to prove our mains results
\begin{lemma} \label{1} \cite{Pas}.
	Let $\mathcal{H}$ be Hilbert $\mathcal{A}$-module. If $T\in End_{\mathcal{A}}^{\ast}(\mathcal{H})$, then $$\langle Tx,Tx\rangle\leq\|T\|^{2}\langle x,x\rangle, \forall x\in\mathcal{H}.$$
\end{lemma}

\begin{lemma} \label{sb} \cite{Ara}.
Let $\mathcal{H}$ and $\mathcal{K}$ two Hilbert $\mathcal{A}$-modules and $T\in End^{\ast}(\mathcal{H},\mathcal{K})$. Then the following statements are equivalent:
\begin{itemize}
	\item [(i)] $T$ is surjective.
	\item [(ii)] $T^{\ast}$ is bounded below with respect to norm, i.e., there is $m>0$ such that $\|T^{\ast}x\|\geq m\|x\|$ for all $x\in\mathcal{K}$.
	\item [(iii)] $T^{\ast}$ is bounded below with respect to the inner product, i.e., there is $m'>0$ such that $\langle T^{\ast}x,T^{\ast}x\rangle\geq m'\langle x,x\rangle$ for all $x\in\mathcal{K}$.
\end{itemize}
\end{lemma}
\begin{lemma} \label{3} \cite{Deh}.
	Let $\mathcal{H}$ and $\mathcal{K}$ two Hilbert $\mathcal{A}$-modules and $T\in End^{\ast}(\mathcal{H},\mathcal{K})$. Then:
	\begin{itemize}
		\item [(i)] If $T$ is injective and $T$ has closed range, then the adjointable map $T^{\ast}T$ is invertible and $$\|(T^{\ast}T)^{-1}\|^{-1}\leq T^{\ast}T\leq\|T\|^{2}.$$
		\item  [(ii)]	If $T$ is surjective, then the adjointable map $TT^{\ast}$ is invertible and $$\|(TT^{\ast})^{-1}\|^{-1}\leq TT^{\ast}\leq\|T\|^{2}.$$
	\end{itemize}	
\end{lemma}

\section{Continuous $\ast$-g-Frame in Hilbert $C^{\ast}$-Modules}
Let $X$ be a Banach space, $(\Omega,\mu)$ a measure space, and function $f:\Omega\to X$ a measurable function. Integral of the Banach-valued function $f$ has defined Bochner and others. Most properties of this integral are similar to those of the integral of real-valued functions. Because every $C^{\ast}$-algebra and Hilbert $C^{\ast}$-module is a Banach space thus we can use this integral and its properties.

Let $(\Omega,\mu)$ be a measure space, let $U$ and $V$ be two Hilbert $C^{\ast}$-modules, $\{V_{w}: w\in\Omega\}$  is a sequence of subspaces of V, and $End_{\mathcal{A}}^{\ast}(U,V_{w})$ is the collection of all adjointable $\mathcal{A}$-linear maps from $U$ into $V_{w}$.
We define
\begin{equation*}
\oplus_{w\in\Omega}V_{w}=\left\{x=\{x_{w}\}: x_{w}\in V_{w}, \left\|\int_{\Omega}|x_{w}|^{2}d\mu(w)\right\|<\infty\right\}.
\end{equation*}
For any $x=\{x_{w}: w\in\Omega\}$ and $y=\{y_{w}: w\in\Omega\}$, if the $\mathcal{A}$-valued inner product is defined by $\langle x,y\rangle=\int_{\Omega}\langle x_{w},y_{w}\rangle d\mu(w)$, the norm is defined by $\|x\|=\|\langle x,x\rangle\|^{\frac{1}{2}}$, the $\oplus_{w\in\Omega}V_{w}$ is a Hilbert $C^{\ast}$-module.
\begin{definition}
	We call $\{\Lambda_{w}\in End_{\mathcal{A}}^{\ast}(U,V_{w}): w\in\Omega\}$ a continuous $\ast$-g-frame for Hilbert $C^{\ast}$-module $U$ with respect to $\{V_{w}: w\in\Omega\}$ if
	\begin{itemize}
		\item for any $x\in U$, the function $\tilde{x}:\Omega\rightarrow V_{w}$ defined by $\tilde{x}(w)=\Lambda_{w}x$ is measurable;
		\item there exist two strictly nonzero elements $A$ and $B$ in $\mathcal{A}$ such that
		\begin{equation} \label{2.1}
		A\langle x,x\rangle A^{\ast}\leq\int_{\Omega}\langle\Lambda_{w}x,\Lambda_{w}x\rangle d\mu(w)\leq B\langle x,x\rangle B^{\ast}, \forall x\in U.
		\end{equation}
	\end{itemize}
		The elements $A$ and $B$ are called continuous $\ast$-g-frame bounds. If $A=B$ we call this continuous $\ast$-g-frame a continuous tight $\ast$-g-frame, and if $A=B=1_{\mathcal{A}}$ it is called a continuous Parseval $\ast$-g-frame. If only the right-hand inequality of \eqref{2.1} is satisfied, we call $\{\Lambda_{w}: w\in\Omega\}$ a 
		continuous $\ast$-g-Bessel for $U$ with respect to $\{\Lambda_{w}: w\in\Omega\}$ with Bessel bound $B$.
	
\end{definition}
We mentioned that the set of all continuous g-frames in  Hilbert $C^{\ast}$-Modules can be considered as a subset of continuous $\ast$-g-frame. To illustrate this, let $\{\Lambda_{w}: w\in\Omega\}$ be a continuous g-frames for Hilbert $C^{\ast}$-Modules $U$ with real continuous g-frames bounds $A$ and $B$. note that for all $x\in U$,
	\begin{equation*}
(\sqrt{A})1_{\mathcal{A}}\langle x,x\rangle(\sqrt{A})1_{\mathcal{A}}\leq\int_{\Omega}\langle\Lambda_{w}x,\Lambda_{w}x\rangle d\mu(w)\leq (\sqrt{B})1_{\mathcal{A}}\langle x,x\rangle(\sqrt{B})1_{\mathcal{A}}.
\end{equation*}
Therefore, every continuous g-frames in  Hilbert $C^{\ast}$-Modules $U$ with real bounds $A$ and $B$ is a continuous $\ast$-g-frame in $U$ with $\mathcal{A}$-valued bounds $(\sqrt{A})1_{\mathcal{A}}$ and $(\sqrt{B})1_{\mathcal{A}}$.
%\begin{example}
%Let $\Omega:=\mathbb{N}$ and $\mu$ be the counting measure, then a $\ast$-g-frame is a continuous $\ast$-g-frame.
%\end{example}
\begin{theorem} \label{2.3}
	Let $\{\Lambda_{w}\in End_{\mathcal{A}}^{\ast}(U,V_{w}): w\in\Omega\}$ be a continuous $\ast$-g-frame for $U$, with lower and upper bounds $A$ and $B$, respectively. The continuous $\ast$-g-frame transform $T:U\rightarrow\oplus_{w\in\Omega}V_{w}$ defined by: $Tx=\{\Lambda_{w}x: w\in\Omega\}$ is injective, closed range adjointable and $\|T\|\leq\|B\|$. The adjoint operator $T^{\ast}$ is surjective and it given by: $T^{\ast}x=\int_{\Omega}\Lambda_{w}^{\ast}x_{w}d\mu(w)$ where $x=\{x_{w}\}_{w\in\Omega}$.
\end{theorem}
\begin{proof}
	Let $x\in U$, by the definition of a continuous $\ast$-g-frame for $U$ we have
	\begin{equation*}
A\langle x,x\rangle A^{\ast}\leq\int_{\Omega}\langle\Lambda_{w}x,\Lambda_{w}x\rangle d\mu(w)\leq B\langle x,x\rangle B^{\ast}.
	\end{equation*}
	So
	\begin{equation}\label{1}
A\langle x,x\rangle A^{\ast}\leq\langle Tx,Tx\rangle\leq B\langle x,x\rangle B^{\ast}.
	\end{equation}
	If $Tx=0$ then $\langle x,x\rangle=0$ so $x=0$ i.e. $T$ is injective.
	
	We now show that the range of $T$ is closed. Let $\{Tx_{n}\}_{n\in\mathbb{N}}$ be a sequence in the range of $T$ such that $\lim_{n\rightarrow\infty}Tx_{n}=y.$
	
	By \eqref{1} we have, for $n, m\in\mathbb{N}$,
	\begin{equation*}
\|A\langle x_{n}-x_{m},x_{n}-x_{m}\rangle A^{\ast}\|\leq\|\langle T(x_{n}-x_{m}),T(x_{n}-x_{m})\rangle\|=\|T(x_{n}-x_{m})\|^{2}.
	\end{equation*}
	Since $\{Tx_{n}\}_{n\in\mathbb{N}}$ is Cauchy sequence in $U$,
	
	$\|A\langle x_{n}-x_{m},x_{n}-x_{m}\rangle A^{\ast}\|\rightarrow0$, as $n,m\rightarrow\infty.$
	
	Note that for $n, m\in\mathbb{N}$,
	$$\|\langle x_{n}-x_{m},x_{n}-x_{m}\rangle\|=\|A^{-1}A\langle x_{n}-x_{m},x_{n}-x_{m}\rangle A^{\ast}(A^{\ast})^{-1}\|$$

$$\leq\|A^{-1}\|^{2}\|A\langle x_{n}-x_{m},x_{n}-x_{m}\rangle A^{\ast}\|.$$
	Therefore the sequence $\{x_{n}\}_{n\in\mathbb{N}}$ is Cauchy and hence there exists $x\in U$ such that $x_{n}\rightarrow x$ as $n\rightarrow\infty$. Again by \eqref{1} we have $\|T(x_{n}-x)\|^{2}\leq\|B\|^{2}\|\langle x_{n}-x,x_{n}-x\rangle\|$.
	
	Thus $\|Tx_{n}-Tx\|\rightarrow0$ as $n\rightarrow\infty$ implies that $Tx=y$. It concludes that the range of $T$ is closed.
	
	For all $x\in U$, $y=\{y_{w}\}\in\oplus_{w\in\Omega}V_{w}$, we have
	\begin{equation*}
	\langle Tx,y\rangle=\int_{\Omega}\langle\Lambda_{w}x,y_{w}\rangle d\mu(w)=\int_{\Omega}\langle x,\Lambda_{w}^{\ast}y_{w}\rangle d\mu(w)=\left\langle x,\int_{\Omega}\Lambda_{w}^{\ast}y_{w}d\mu(w)\right\rangle.
	\end{equation*}
Then $T$ is adjointable and $T^{\ast}y=\int_{\Omega}\Lambda_{w}^{\ast}y_{w}d\mu(w).$

By \eqref{1} we have $\|Tx\|^{2}\leq\|B\|^{2}\|x\|^{2}$ so $\|T\|\leq\|B\|.$

By \eqref{1} we have $\|Tx\|\geq\|A^{-1}\|^{-1}\|x\|$ $\forall x\in U$ so by lemma \ref{sb} $T^{\ast}$ is surjective.
This completes the proof.
\end{proof}
Now we define the continuous $\ast$-g-frame operator and studies some of its properties.
\begin{definition}
Let $\{\Lambda_{w}\in End_{\mathcal{A}}^{\ast}(U,V_{w}): w\in\Omega\}$ be a continuous $\ast$-g-frame for $U$. Define the continuous $\ast$-g-frame operator $S$ on $U$ by: $Sx=T^{\ast}Tx=\int_{\Omega}\Lambda_{w}^{\ast}\Lambda_{w}xd\mu(w)$.
\end{definition}
\begin{theorem}
The continuous $\ast$-g-frame operator $S$ is a bounded, positive, selfadjoint, invertible and $\|A^{-1}\|^{-2}\leq\|S\|\leq\|B\|^{2}$.
\end{theorem}
\begin{proof}
	First we show, $S$ is a selfadjoint operator. By definition we have $\forall x, y\in U$
	\begin{align*}
	\langle Sx,y\rangle&=\left\langle\int_{\Omega}\Lambda_{w}^{\ast}\Lambda_{w}xd\mu(w),y\right\rangle\\
	&=\int_{\Omega}\langle\Lambda_{w}^{\ast}\Lambda_{w}x,y\rangle d\mu(w)\\
	&=\int_{\Omega}\langle x,\Lambda_{w}^{\ast}\Lambda_{w}y\rangle d\mu(w)\\
	&=\left\langle x,\int_{\Omega}\Lambda_{w}^{\ast}\Lambda_{w}yd\mu(w)\right\rangle\\
	&=\langle x,Sy\rangle.
	\end{align*}
	Then $S$ is a selfadjoint.
	
	By Lemma \ref{3} and Theorem \ref{2.3}, $S$ is invertible. Clearly $S$ is positive.
	
	By definition of a continuous $\ast$-g-frame we have
	\begin{equation*}
	A\langle x,x\rangle A^{\ast}\leq\int_{\Omega}\langle\Lambda_{w}x,\Lambda_{w}x\rangle d\mu(w)\leq B\langle x,x\rangle B^{\ast}.
	\end{equation*}
	So
	\begin{equation*}
	A\langle x,x\rangle A^{\ast}\leq\langle Sx,x\rangle\leq B\langle x,x\rangle B^{\ast}.
	\end{equation*}
	This give
	\begin{equation*}
	\|A^{-1}\|^{-2}\|x\|^{2}\leq\|\langle Sx,x\rangle\|\leq\|B\|^{2}\|x\|^{2}, \forall x\in U.
	\end{equation*}
	If we take supremum on all $x\in U$, where $\|x\|\leq1$, then $\|A^{-1}\|^{-2}\leq\|S\|\leq\|B\|^{2}$.
\end{proof}
\begin{theorem}
Let $\{\Lambda_{w}\in End_{\mathcal{A}}^{\ast}(U,V_{w}): w\in\Omega\}$ be a continuous $\ast$-g-frame for $U$, with lower and upper bounds $A$ and $B$, respectively and with the continuous $\ast$-g-frame operator $S$. Let $T\in End_{\mathcal{A}}^{\ast}(U)$ be invertible. Then $\{\Lambda_{w}T: w\in\Omega\}$ is a continuous $\ast$-g-frame for $U$ with continuous $\ast$-g-frame operator $T^{\ast}ST$.
\end{theorem}
\begin{proof}
	We have 
	\begin{equation}\label{eq11}
	A\langle Tx,Tx\rangle A^{\ast}\leq\int_{\Omega}\langle\Lambda_{w}Tx,\Lambda_{w}Tx\rangle d\mu(w)\leq B\langle Tx,Tx\rangle B^{\ast}, \forall x\in U.
	\end{equation}
	Using Lemma , we have $\|(T^{\ast}T)^{-1}\|^{-1}\langle x,x\rangle\leq\langle Tx,Tx\rangle$, $\forall x\in U$. Or $\|T^{-1}\|^{-2}\leq\|(T^{\ast}T)^{-1}\|^{-1}$. This implies
	\begin{equation}\label{eq22} 
		\|T^{-1}\|^{-1}A\langle x,x\rangle(\|T^{-1}\|^{-1}A)^{\ast}\leq A\langle Tx,Tx\rangle A^{\ast}, \forall x\in U.
	\end{equation}
	And we know that $\langle Tx,Tx\rangle\leq\|T\|^{2}\langle x,x\rangle$, $\forall x\in U$. This implies that
	\begin{equation}\label{eq33}
	B\langle Tx,Tx\rangle B^{\ast}\leq\|T\|B\langle x,x\rangle(\|T\|B)^{\ast}, \forall x\in U.
	\end{equation}
Using \eqref{eq11}, \eqref{eq22}, \eqref{eq33} we have
\begin{equation}
\|T^{-1}\|^{-1}A\langle x,x\rangle(\|T^{-1}\|^{-1}A)^{\ast}\leq\int_{\Omega}\langle\Lambda_{w}Tx,\Lambda_{w}Tx\rangle d\mu(w)\leq B\|T\|\langle x,x\rangle(B\|T\|)^{\ast}, \forall x\in U.
\end{equation}
So $\{\Lambda_{w}T: w\in\Omega\}$ is a continuous $\ast$-g-frame for $U$.

Moreover for every $x\in U$, we have
$$T^{\ast}STx=T^{\ast}\int_{\Omega}\Lambda_{w}^{\ast}\Lambda_{w}Txd\mu(w)=\int_{\Omega}T^{\ast}\Lambda_{w}^{\ast}\Lambda_{w}Txd\mu(w)=\int_{\Omega}(\Lambda_{w}T)^{\ast}(\Lambda_{w}T)xd\mu(w).$$ This completes the proof.	
\end{proof}
\begin{corollary}
	Let $\{\Lambda_{w}\in End_{\mathcal{A}}^{\ast}(U,V_{w}): w\in\Omega\}$ be a continuous $\ast$-g-frame for $U$, with continuous $\ast$-g-frame operator $S$. Then $\{\Lambda_{w}S^{-1}: w\in\Omega\}$ is a continuous $\ast$-g-frame for $U$.
\end{corollary}
\begin{proof}
	Result the next theorem by taking $T=S^{-1}$.
\end{proof}	
\section{Stability problem for Continuous $\ast$-g-Frame in Hilbert $C^{\ast}$-Modules}
The question of stability plays an important role in various
fields of applied mathematics. The classical theorem of
the stability of a base is due to Paley and Wiener. It is based on the fact that a bounded operator T on a Banach space is invertible if we have: $\|I-T\|<1$.
\begin{theorem} [\cite{Paley} Paley-Wiener]
	Let $\{f_{i}\}_{i\in\mathbb{N}}$ be a basis of a Banach space $X$, and $\{g_{i}\}_{i\in\mathbb{N}}$ a sequence of vectors in $X$. If there exists a constant $\lambda\in[0,1)$ such that
	\begin{equation*}
		\Big\|\sum_{i\in\mathbb{N}}c_{i}(f_{i}-g_{i})\Big\|\leq\lambda\Big\|\sum_{i\in\mathbb{N}}c_{i}f_{i}\Big\|
	\end{equation*}
	for all finite sequence  $\{c_{i}\}_{i\in\mathbb{N}}$ of scalars, then $\{g_{i}\}_{i\in\mathbb{N}}$ is also a basis for $X$.
\end{theorem}
\begin{theorem}
Let $\{\Lambda_{w}\in End_{\mathcal{A}}^{\ast}(U,V_{w}): w\in\Omega\}$ be a continuous $\ast$-g-frame for $U$, with lower and upper bounds $A$ and $B$, respectively. Let $\Gamma_{w}\in End_{\mathcal{A}}^{\ast}(U,V_{w})$ for any $w\in\Omega$. Then the following are equivalent:
\begin{itemize}
	\item [(1)] $\{\Gamma_{w}\in End_{\mathcal{A}}^{\ast}(U,V_{w}): w\in\Omega\}$ is a continuous $\ast$-g-frame for $U$.
	\item[(2)] There exists a constant $M>0$, such that for any $x\in U$, one has
	\begin{multline}\label{3.1}
		\bigg\|\int_{\Omega}\langle(\Lambda_{w}-\Gamma_{w})x,(\Lambda_{w}-\Gamma_{w})x\rangle d\mu(w)\bigg\|\\
		\leq M\min\bigg(\bigg\|\int_{\Omega}\langle\Lambda_{w}x,\Lambda_{w}x\rangle d\mu(w)\bigg\|,\bigg\|\int_{\Omega}\langle\Gamma_{w}x,\Gamma_{w}x\rangle d\mu(w)\bigg\|\bigg).
	\end{multline}
\end{itemize}
\end{theorem}
\begin{proof}
	$(1)\Rightarrow(2)$. Suppose that $\{\Gamma_{w}\in End_{\mathcal{A}}^{\ast}(U,V_{w}): w\in\Omega\}$ is a continuous $\ast$-g-frame for $U$ with lower and upper bounds $C$ and $D$, respectively. Then for any $x\in U$, we have
	\begin{multline*}
\bigg\|\int_{\Omega}\langle(\Lambda_{w}-\Gamma_{w})x,(\Lambda_{w}-\Gamma_{w})x\rangle d\mu(w)\bigg\|^{\frac{1}{2}}=\big\|\{(\Lambda_{w}-\Gamma_{w})x\}_{w\in\Omega}\big\|\\
\leq\big\|\{\Lambda_{w}x\}_{x\in\Omega}\big\|+\big\|\{\Gamma_{w}x\}_{x\in\Omega}\big\|\\
=\bigg\|\int_{\Omega}\langle\Lambda_{w}x,\Lambda_{w}x\rangle d\mu(w)\bigg\|^{\frac{1}{2}}+\bigg\|\int_{\Omega}\langle\Gamma_{w}x,\Gamma_{w}x\rangle d\mu(w)\bigg\|^{\frac{1}{2}}\\
\leq\|B\|\|\langle x,x\rangle\|^{\frac{1}{2}}+\bigg\|\int_{\Omega}\langle\Gamma_{w}x,\Gamma_{w}x\rangle d\mu(w)\bigg\|^{\frac{1}{2}}\\
\leq\|B\|\|C^{-1}\|\bigg\|\int_{\Omega}\langle\Gamma_{w}x,\Gamma_{w}x\rangle d\mu(w)\bigg\|^{\frac{1}{2}}+\bigg\|\int_{\Omega}\langle\Gamma_{w}x,\Gamma_{w}x\rangle d\mu(w)\bigg\|^{\frac{1}{2}}\\
=\bigg(\|B\|\|C^{-1}\|+1\bigg)\bigg\|\int_{\Omega}\langle\Gamma_{w}x,\Gamma_{w}x\rangle d\mu(w)\bigg\|^{\frac{1}{2}}.
	\end{multline*}
	Similary we have
	\begin{align*}
\bigg\|\int_{\Omega}\langle(\Lambda_{w}-\Gamma_{w})x,(\Lambda_{w}-\Gamma_{w})x\rangle d\mu(w)\bigg\|^{\frac{1}{2}}\leq\bigg(\|D\|\|A^{-1}\|+1\bigg)\bigg\|\int_{\Omega}\langle\Lambda_{w}x,\Lambda_{w}x\rangle d\mu(w)\bigg\|^{\frac{1}{2}}.
	\end{align*}
Let $M=\min\Bigg\{\bigg(\|B\|\|C^{-1}\|+1\bigg)^{2},\bigg(\|D\|\|A^{-1}\|+1\bigg)^{2}\Bigg\}$, then the inequality \eqref{3.1} holds.

$(2)\Rightarrow(1)$. Suppose that the inequality \eqref{3.1} holds. For any $x\in U$, we have
	\begin{multline*}
\|A^{-1}\|^{-1}\|\langle x,x\rangle\|^{\frac{1}{2}}\leq\bigg\|\int_{\Omega}\langle\Lambda_{w}x,\Lambda_{w}x\rangle d\mu(w)\bigg\|^{\frac{1}{2}}\\
\leq\bigg\|\int_{\Omega}\langle(\Lambda_{w}-\Gamma_{w})x,(\Lambda_{w}-\Gamma_{w})x\rangle d\mu(w)\bigg\|^{\frac{1}{2}}+\bigg\|\int_{\Omega}\langle\Gamma_{w}x,\Gamma_{w}x\rangle d\mu(w)\bigg\|^{\frac{1}{2}}\\
\leq M^{\frac{1}{2}}\bigg\|\int_{\Omega}\langle\Gamma_{w}x,\Gamma_{w}x\rangle d\mu(w)\bigg\|^{\frac{1}{2}}+\bigg\|\int_{\Omega}\langle\Gamma_{w}x,\Gamma_{w}x\rangle d\mu(w)\bigg\|^{\frac{1}{2}}\\
=\big(1+M^{\frac{1}{2}}\big)\bigg\|\int_{\Omega}\langle\Gamma_{w}x,\Gamma_{w}x\rangle d\mu(w)\bigg\|^{\frac{1}{2}}.
\end{multline*}
Also we obtain
\begin{multline*}
\bigg\|\int_{\Omega}\langle\Gamma_{w}x,\Gamma_{w}x\rangle d\mu(w)\bigg\|^{\frac{1}{2}}\\
\leq\bigg\|\int_{\Omega}\langle(\Lambda_{w}-\Gamma_{w})x,(\Lambda_{w}-\Gamma_{w})x\rangle d\mu(w)\bigg\|^{\frac{1}{2}}+\bigg\|\int_{\Omega}\langle\Lambda_{w}x,\Lambda_{w}x\rangle d\mu(w)\bigg\|^{\frac{1}{2}}\\
\leq M^{\frac{1}{2}}\bigg\|\int_{\Omega}\langle\Lambda_{w}x,\Lambda_{w}x\rangle d\mu(w)\bigg\|^{\frac{1}{2}}+\bigg\|\int_{\Omega}\langle\Lambda_{w}x,\Lambda_{w}x\rangle d\mu(w)\bigg\|^{\frac{1}{2}}\\
=\big(1+M^{\frac{1}{2}}\big)\bigg\|\int_{\Omega}\langle\Lambda_{w}x,\Lambda_{w}x\rangle d\mu(w)\bigg\|^{\frac{1}{2}}\\
\leq\big(1+M^{\frac{1}{2}}\big)\|B\|\|\langle x,x\rangle\|^{\frac{1}{2}}.
\end{multline*}
So $\{\Gamma_{w}\in End_{\mathcal{A}}^{\ast}(U,V_{w}): w\in\Omega\}$ is a continuous $\ast$-g-frame for $U$.
\end{proof}
%\section*{Acknowledgment}
%It is our great pleasure to thank the referee for his careful reading of the paper and for several helpful suggestions.
\bibliographystyle{amsplain}

\end{document}